\documentclass[12pt,twoside,reqno]{amsart}

\usepackage{tikz}
\usepackage{amsmath}
\usepackage{amssymb,amsfonts,mathrsfs}
\usepackage[colorlinks=true,linkcolor=blue,citecolor=blue]{hyperref}

\usepackage{mllocal} % local macro definitions, just to avoid the clutter

\numberwithin{equation}{section}

\begin{document}

\title[The exotic heat-trace asymptotics revisited]
{The exotic heat-trace asymptotics of a \\ regular-singular operator revisited}

\author{Boris Vertman}
\address{Mathematisches Institut,
Universit\"at Bonn,
53115 Bonn,
Germany}
\email{vertman@math.uni-bonn.de}
\urladdr{www.math.uni-bonn.de/people/vertman}

\date{This document compiled on: \today.}

\begin{abstract}
{We discuss the exotic properties of the heat-trace 
asymptotics for a regular-singular operator with general boundary conditions at 
the singular end, as observed by Falomir, Muschietti, Pisani and Seeley in \cite{FMPS:UPZ}
as well as by Kirsten, Loya and Park in \cite{KLP:UPR}. We explain how their results alternatively follow from 
the general heat kernel construction by Mooers \cite{Moo:HKA}, 
a natural question that has not been addressed yet, as the latter work did not elaborate
explicitly on the singular structure of the heat trace expansion beyond the statement of 
non-polyhomogeneity of the heat kernel.}
\end{abstract}

\maketitle

%%%%%%%%%%%%%%%%%%
\section{Introduction}
%%%%%%%%%%%%%%%%%%

In this paper we revisit the discussion of the heat kernel and the unusual properties of the heat trace expansion 
for a general self-adjoint realization in $L^2(0,1)$ of the regular-singular operator ($\R^+=(0,\infty)$)
$$
\Delta_\nu = -\frac{d^2}{dx^2} + \frac{1}{x^2}\left(\nu^2-\frac{1}{4}\right): C^\infty_0(\R^+) \to C^\infty_0(\R^+),
\quad \nu \geq 0,
$$
where $C^\infty_0(\R^+)$ denotes the space of smooth functions with compact support in $\R^+$.
\medskip

The heat trace expansion of $\Delta_\nu$ for $\nu >0$ is well-understood in 
\cite{KLP:EEA} and in \cite{Moo:HKA}, and in fact does not exhibit new phenomena, 
compare in particular the general discussion of Gil, Krainer and Mendoza \cite{GKM:TEF}.
The intricate case is rather $\nu=0$, where the relation between the explicit approaches of 
\cite{FMPS:UPZ}, \cite{KLP:UPR} and the methods of the heat kernel construction in 
\cite{Moo:HKA} is less obvious.\medskip

The regular-singular operators $\Delta_\nu$ arise naturally in the spectral geometry of spaces with isolated 
conical singularities, modelled by a bounded generalized cone 
$(\mathscr{C}(N)= \R^+ \times N, g=dx^2 \oplus x^2 g_N)$
over a closed Riemannian manifold $(N,g^N)$ of dimension $n$. This includes the example of 
a higher-dimensional disc $\mathscr{D}\subset \R^{n+1}$, where the Euclidean metric takes the 
form $dx^2 \oplus x^2 g_{S^{n}}$ with respect to standard polar coordinates. 
The Laplace-Beltrami operator on $(\mathscr{C}(N), g)$ is a symmetric operator 
in $L^2(\mathscr{C}(N), g)$ and takes the form
$$
\Delta_{\mathscr{C}(N)} = -\partial_x^2 - \frac{n}{x} \, \partial_x + x^{-2}\Delta_N.
$$
Under the unitary transformation (see \cite{BruSee:RES})

$$\Phi: L^2(\mathscr{C}(N), g) \to L^2(\R^+, L^2(N,g_N), dx), 
\, \w \mapsto x^{n/2} \w, $$ 

the Laplacian $\Delta_{\mathscr{C}(N)}$ unitarily transforms to 
$$
\Phi \circ \Delta_{\mathscr{C}(N)} \circ \Phi^{-1} = 
-\partial_x^2 + \frac{1}{x^2}\left( \Delta_N + \frac{(n-1)^2}{4}- \frac{1}{4} \right)=:\Delta^{\Phi}_{\mathscr{C}(N)}.
$$
The spectral decomposition on the cross section $(N,g_N)$ decomposes $\Delta^{\Phi}_{\mathscr{C}(N)}$
into a direct sum of operators $\Delta_\nu$ over the $\nu^2$-eigenspaces
of $\Delta_N$.  We are interested here in the operator over the kernel of $\Delta_N$ (for $n=1$)
$$
\Delta = -\frac{d^2}{dx^2} - \frac{1}{4x^2}: C^\infty_0(\R^+) \to C^\infty_0(\R^+).
$$

The spectral properties of $\Delta$ for general self-adjoint boundary 
conditions have been subject of careful analysis by Falomir, Muschietti, Pisani and Seeley in \cite{FMPS:UPZ}
as well as by Kirsten, Loya and Park in \cite{KLP:UPR}. In both instances 
the authors uncovered new unsual phenomena in the heat trace expansion 
of $\Delta$, if the boundary conditions at $x=0$ do not define 
the Friedrichs extension. Their (and our) result reads as follows.

\begin{theorem}\label{main}
The self-adjoint boundary conditions of $\Delta$ at $x=0$ are parametrized
by $\theta \in [0,\pi)$, where $\theta=\pi/2$ corresponds to the Friedrichs extension.
We write $\Delta(\theta)$ for the corresponding self-adjoint realization. With 
$$\kappa_\theta := \gamma -\log 2 + \tan \theta,$$ where $\gamma$ denotes 
the Euler-Mascheroni constant, we find for $\theta\neq \pi/2$
\begin{align*}
\textup{Tr}_1 e^{-t\Delta(\theta)} &:= \int_0^1 e^{-t\Delta(\theta)}(x,x) \, dx \sim 
\textup{Tr}_1 e^{-t\Delta(\pi/2)} + \sum_{j = 0}^\infty a_j t^{\,j} \\
& + \frac{1}{\pi} \textup{Im} \left(\int_1^\infty \frac{e^{-ty}}{y}
(\log y + i\pi + 2 \kappa_\theta)^{-1} dy \right), \ t\to 0.
\end{align*}
\end{theorem}

Note the absense of fractional $t$-powers due to missing contribution from the regular boundary in the present setup.

Both discussions of \cite{FMPS:UPZ} and \cite{KLP:UPR} have been performed independently from the earlier work by 
Mooers \cite{Moo:HKA}, who in particular constructed the heat kernel of $\Delta$
for a general self-adjoint realization and observed its non-polyhomogeneity
in case of a non-Friedrichs extension at $x=0$. \medskip

However, \cite{Moo:HKA} did not elaborate further on the particular 
structure of the non-polyhomogeneous heat kernel and its heat trace 
behaviour, leading to the natural question of whether her analysis
can be reconciled with the explicit results in \cite{FMPS:UPZ}
and \cite{KLP:UPR}. Hereby, certain computational inconsistencies in 
\cite{Moo:HKA} become apparent, see Remark \ref{error}.\medskip

In this note we address this issue and show how the 
heat trace expansion results in \cite{KLP:UPR} follow straightforwardly from
the heat kernel construction of Mooers in \cite{Moo:HKA}. At various points we 
chose to provide arguments, alternative to \cite{Moo:HKA}.
\medskip

This paper is organized as follows. We first classify the 
self-adjoint realizations for $\Delta$ in \S \ref{self-adjoint-section}, and study the solution 
to the signaling problem in \S \ref{signal-section}.
The solution to the signaling problem is the central ingredient 
in the construction of the heat kernel for general self-adjoint 
boundary conditions, which is explained in \S \ref{heat-section}
and is basically a revision of \cite{Moo:HKA}. 
In \S \ref{auxiliary} we establish well-definement of the various 
expressions in the general heat kernel formula. In the final \S 
\ref{unusual-section} we derive the heat trace expansion 
directly from the heat kernel structure in \S \ref{heat-section}, thus
reproving Theorem \ref{main} by a different method.

\section*{Acknowledgements}
The author would like to thank Thomas Krainer for helpful remarks and gratefully 
acknowledges the support by the Hausdorff Research Institute at
the University of Bonn.

\tableofcontents

%%%%%%%%%%%%%%%%%%
\section{Self-adjoint boundary conditions}\label{self-adjoint-section}
%%%%%%%%%%%%%%%%%%

Consider the differential operator 

\begin{align*}
\Delta = -\frac{d^2}{dx^2} - \frac{1}{4x^2}: C^\infty_0(\R^+) \to C^\infty_0(\R^+).
\end{align*}

acting on $C^\infty_0(\R^+) \subset L^2(\R^+)$. We always denote by $C^\infty_0(I)$ 
the space of smooth functions with compact support in $I\subset \R$. We define 
the \emph{minimal} closed extension $\Delta_{\min}$ to be the graph closure of 
$\Delta$, which is a densely defined and symmetric operator. The \emph{maximal}
closed extension is $\Delta_{\max} := (\Delta_{\min})^*$. 
\medskip

We then have the following well-known characterization of the maximal domain $\mathscr{D}(\Delta_{\max})$,
written out in various sources, including \cite{BruSee:RES}, \cite{Moo:HKA}, \cite{FMPS:UPZ},\cite{KLP:UPR},
\cite{LesVer:RSS} and \cite{Ver:ZDF}.

\begin{prop}\label{max-domain}
Let $f \in \mathscr{D}(\Delta_{\max})$. Then there exist constants 
$c_+(f), c_-(f)\in \C$ depending only on $f$, and a continuously differentiable $\widetilde{f}\in \mathscr{D}(\Delta_{\min})$, 
with $\widetilde{f}(x)=O(x^{3/2}\log(x))$ and $\widetilde{f}'(x)=O(x^{1/2}\log(x))$, as $x\to 0$,
such that 
$$
f(x)=c_+(f) \sqrt{x} + c_-(f)\sqrt{x}\log(x) + \widetilde{f}(x).
$$
Moreover, for any $f,g\in \mathscr{D}(\Delta_{\max})$ the following Green's identity holds
$$
\langle \Delta f, g \rangle - \langle f, \Delta g \rangle = 
\overline{c_-(f)}c_+(g) - \overline{c_+(f)}c_-(g)
$$
\end{prop} \ \medskip

An extension $\mathscr{D}(\Delta_{\min}) \subseteq \mathscr{D} \subseteq \mathscr{D}(\Delta_{\max})$
is self-adjoint, if 
$$
\mathscr{D} = \{ f\in \mathscr{D}(\Delta_{\max}) \mid \forall g\in \mathscr{D}: 
\langle \Delta f, g \rangle_{L^2} = \langle f, \Delta g \rangle_{L^2} \}.
$$

Consequently we obtain a self-adjoint extension of $\Delta$ by choosing the 
following boundary operators on $f\in \mathscr{D}(\Delta_{\max})$
\begin{equation}
B_\theta(f) := \cos \theta \cdot c_+(f) + \sin \theta \cdot c_-(f), 
\quad \theta \in [0,\pi).
\end{equation}
We define then
\begin{align}
\dom (\Delta(\theta)):= \{f\in \mathscr{D}(\Delta_{\max}) 
\mid B_\theta(f)=0\}.
\end{align}
The extensions $\Delta(\theta), \theta \in [0,\pi)$
define self-adjoint realizations of $\Delta$ by Proposition \ref{max-domain},
and in fact classify completely all self-adjoint boundary conditions at $x=0$.

%%%%%%%%%%%%%%%%%%
\section{The signaling solution}\label{signal-section}
%%%%%%%%%%%%%%%%%%

The fundamental component in the heat kernel construction 
of Mooers in \cite{Moo:HKA} is the \emph{signaling solution} $F(h)(\cdot, t)\in \dom(\Delta_{\max}), t\in \R^+_0$,
with $\R^+_0=[0,\infty)$, defined for any given $h\in L^1_{\textup{loc}}(\R^+_0) \cap C^\infty(\R^+)$, 
as a solution to the so-called \emph{signaling problem}

\begin{equation}
\begin{split}
(\partial_t + \Delta) F(h)(x,t)&=0, \quad F(h)(\cdot,0)\equiv 0, \\
c_-(F(h)(\cdot, t))&=h(t),\ t>0.
\end{split}
\end{equation}

Note that there does not exist $\theta \in [0,\pi)$, such that $F(h)(\cdot, t)\in \dom(\Delta(\theta))$ for 
$t\in \R^+_0$, since by uniqueness of solutions to 
the heat equation, $F(h)(\cdot,0)\equiv 0$ then implies $F(h)\equiv 0$. 
The signaling solution is in fact unique, since for $c_-(F(h)(\cdot, t))=0$, 
$F(h)(\cdot, t)\in \dom(\Delta(\pi/2))$ and hence 
$F(h)(\cdot,0)\equiv 0$ then implies $F(h)\equiv 0$.
\medskip

Let the differential expression $\Delta$ act on $C_0^\infty(\R^+)$ and 
consider the heat kernel $E_+(x,\wx,t)$ of the Friedrichs extension of $\Delta$
in $L^2(\R^+)$. We write $N'E_+(x,t):= \wx^{-1/2}E_+(x,\wx,t) \restriction \wx=0$. 
By \cite{Les:OFT} and the asymptotics of the modified Bessel functions,
see \cite{AbrSte:HOM}

$$I_{\nu}(z)\sim \frac{z^{\nu}}{2^{\nu}\Gamma (\nu+1)}, \quad 
\textup{as $z\to 0$},$$
we have

\begin{equation}
\label{Friedrichs-heat}
\begin{split}
E_+(x,\wx,t) &= \frac{\sqrt{x\wx}}{2t} I_0\left(\frac{x\wx}{2t}\right) 
\exp\left(- \frac{x^2+\wx^2}{4t}\right), \\
N'E_+(x,t) &= \frac{\sqrt{x}}{2t} \exp\left(- \frac{x^2}{4t}\right).
\end{split}
\end{equation}

\begin{lemma} \cite[Proposition 4.2]{Moo:HKA} \\
For any $h\in L^1_{\textup{loc}}(\R^+_0) \cap C^\infty(\R^+)$ the signaling solution is given by
$$F(h)(x,t) := - \int_0^t h(t-\wt) N'E_+(x,\wt) d\wt.$$
\end{lemma}

\begin{proof}
\cite[Proposition 4.2]{Moo:HKA} establishes the statement by deriving the formula
for $F(h)$ conceptually. Here, for purposes of brevity, we simply show by a direct computation that the solution
$F(h)$ above indeed provides the signaling solution.
With $N'E_+$ solving the heat equation, clearly $(\partial_t + \Delta) F(h)(x,t)=0$.
\medskip

By the explicit formulas \eqref{Friedrichs-heat} we find for any fixed $t>0$, 
$F(h)(\cdot, t) \in L^2(\R^+)$. Similarly, differentiating $F(h)(x,t)$ explicitly in the
first argument, we find $\Delta F(h)(\cdot, t) \in L^2(\R^+)$ and consequently, 
$F(h)(\cdot, t) \in \dom (\Delta_{\max})$. \medskip

It remains to identify $c_-(F(h)(\cdot, t))$. We write $h(t-\wt)=:h(t)-\wt \mathscr{H}(t-\wt),$
where $\mathscr{H}\in C^\infty(\R^+)$ by smoothness of $h$. Then, substituting $T=x^2/\wt$
in the integrals below, we find
\begin{align*}
&F(h)(x,t) = - \int_0^t h(t-\wt) N'E_+(x,\wt) d\wt \\&= 
- \int_{x^2/t}^\infty h\left(t-\frac{x^2}{T}\right) \frac{\sqrt{x}}{2T} \exp\left(-\frac{T}{4}\right) dT \\
&=- h(t) \sqrt{x} \int_{x^2/t}^\infty \frac{1}{2T} \exp\left(-\frac{T}{4}\right) dT  
\\&+ (\sqrt{x})^5 \int_{x^2/t}^\infty \mathscr{H}\left(t-\frac{x^2}{T}\right) \frac{1}{2T^2} \exp\left(-\frac{T}{4}\right) dT\\
&=:I_1 + I_2.
\end{align*}
For the first integral and $x^2<t$ (note $t>0$ is fixed) we obtain 
\begin{align*}
I_1&= - h(t) \sqrt{x}  \int\limits_{x^2/t}^1 \sum_{k=0}^\infty \frac{(-1)^kT^{k-1}}{2^{2k+1}k!} dT 
- h(t) \sqrt{x} \int\limits_{1}^\infty \frac{1}{2T} \exp\left(-\frac{T}{4}\right) dT \\
&= h(t) \sqrt{x} \log(x) + O(\sqrt{x}), \ x\to 0. 
\end{align*}
For the second integral $I_2$ we obtain similarly $I_2=O(\sqrt{x})$ as $x \to 0$, 
and hence indeed find $c_-(F(h)(\cdot, t))=h(t)$.
\end{proof}

It should be noted that the naive expansion of $F(h)$ does not lead to a closed
expression of $c_+(F(h))$, which requires rather integral transform arguments and 
is the fundamental aspect of the next section.

%%%%%%%%%%%%%%%%%%
\section{Heat kernel for general boundary conditions}\label{heat-section}
%%%%%%%%%%%%%%%%%%

We can now discuss the heat kernel for a self-adjoint 
operator
\begin{equation*}
\dom(\Delta(\theta)) = \{f\in \dom(\Delta_{\max}) \!\mid \!
B_\theta(f) := \cos \theta \, c_+(f) + \sin \theta \,  c_-(f)=0\},
\end{equation*}
with  $\theta \in [0,\pi)$.
By the Duhamel principle, the heat trace expansion of 
any self-adjoint extension of $\Delta$ in $L^2(0,1)$ with boundary conditions
$B_\theta$ at $x=0$ and separated boundary conditions at $x=1$,
differs from the heat trace expansion of $\Delta(\theta)$ (trace integral taken over $[0,1]$)
by the classical trace contribution from the regular boundary $x=1$ and $O(t^\infty), t\to 0$,
cf. for instance Lesch \cite[Theorem 1.4.11]{Les:OFT}.
\medskip

The heat kernel $E_+$ for the Friedrichs extension $\Delta(\pi/2)$ is well-known with 
the explicit expression written out in \eqref{Friedrichs-heat}. Fix any $\theta \in [0,\pi)\backslash \{\pi/2\}$. 
We do not distinguish notationally between the pseudodifferential operators and their Schwartz kernels.
Consider $\phi \in C^\infty_0(\R^+_0)$, vanishing to infinite order as $x\to 0$, and put $u=E_+\phi$.  
We seek to correct $u$ by $F(h)$ for an appropriate $h\in L^1_{\textup{loc}}(\R^+_0) \cap C^\infty(\R^+)$, to satisfy 
the boundary conditions $B_\theta(u+F(h))=0$. Note 
\begin{equation}
\label{correction}
\begin{split}
\w &:= u + F(h) \in \dom (\Delta(\theta)) \\ 
&\Leftrightarrow c_+(u) +c_+(F(h)) = -\tan (\theta) \,  c_-(F(h))\\
&\Leftrightarrow N'E_+\phi + G *_t h(t) = -\tan (\theta)\, h(t),
\end{split}
\end{equation}
where $G$ is the convolution kernel mapping $h\equiv c_-(F(h))$ to $c_+(F(h))$.
We make $G(t)$ explicit using the Laplace transform. For any $g\in L^1_{\textup{loc}}(\R^+_0) \cap C^\infty(\R^+)$, 
not growing exponentially as $t\to\infty$, the Laplace transform is defined as follows
\begin{align}\label{laplace}
(\mathscr{L}g)(\zeta) = \int_{\R^+} g(t)\exp(-\zeta t) dt, \textup{Re}(\zeta) >0.
\end{align}
The inverse Laplace transform is given for any $\delta >0$ and analytic $L(\zeta)$, 
integrable over $\textup{Re}(\zeta) = \delta$, by
\begin{align}\label{laplace-inverse}
(\mathscr{L}^{-1}L)(t) = \frac{1}{2\pi i}\int_{\delta + i\, \R} e^{t\zeta} L(\zeta) \, d\zeta.
\end{align}

\begin{lemma}\label{G}
\begin{align}
\mathscr{L}G(\zeta) = \log \sqrt{\zeta} + \gamma - \log 2.
\end{align}
\end{lemma}

\begin{proof}
We compute, cf. \cite[p.31]{Moo:HKA}
\begin{align*}
(\mathscr{L}F(h))(x,\zeta) &= - (\mathscr{L}h)(\zeta) \cdot (\mathscr{L}N'E_+)(x,\zeta)
\\ &= - \sqrt{x} K_0(x\sqrt{\zeta})  (\mathscr{L}h)(\zeta),\quad  \textup{Re}(\zeta)>0
\end{align*}
where $K_0$ is the modified Bessel function of second kind, and in the definition of $\sqrt{\zeta}$
we fix the branch of logarithm in $\C\backslash \R^-$. We assume that $h(t)$ is not of exponential 
growth as $t\to\infty$, so that $(\mathscr{L}h)(\zeta)$ is well-defined for $\textup{Re}(\zeta)>0$. The Bessel function $K_0(z)$ admits 
an asymptotic expansion, see \cite{AbrSte:HOM} 
\begin{align}
K_0(z) \sim - \log(z) + (\log 2- \gamma ) + \widetilde{K}_0(z), \ \widetilde{K}_0(z)=O(z), \ z\to 0,
\end{align}
where $\gamma\in \R$ is the Euler constant. Consequently 
\begin{equation}
\begin{split}
c_+((\mathscr{L}F(h))(\cdot,\zeta)) &=  - (\mathscr{L}h)(\zeta)  (\log 2 -\log \sqrt{\zeta} - \gamma), \\
c_-((\mathscr{L}F(h))(\cdot,\zeta)) &=  (\mathscr{L}h)(\zeta).
\end{split}
\end{equation}
Taking the inverse Laplace transform, we obtain
\begin{align*}
F(h)(x,t) &=\sqrt{x}\log(x)\mathscr{L}^{-1}(c_-(\mathscr{L}F(h))) + 
\sqrt{x}\mathscr{L}^{-1}(c_+(\mathscr{L}F(h))) \\ &+ 
\sqrt{x}\mathscr{L}^{-1}((\mathscr{L}h)\widetilde{K}_0(x\sqrt{\cdot})),
\end{align*}
where each $\mathscr{L}^{-1}(c_{\pm}(\mathscr{L}F(h)))$ exists and 
$\mathscr{L}^{-1}((\mathscr{L}h)\widetilde{K}_0(x\sqrt{\cdot}))=O(x^{3/2})$,
uniformly as $x\to 0$. Consequently, indeed 
$$\mathscr{L}(c_{\pm}(F(h))) =c_{\pm} (\mathscr{L}F(h)).$$
This yields the following explicit expression for the Laplace transform of $G$
\begin{align}\label{LG}
(\mathscr{L}G)(\zeta) =  \log \sqrt{\zeta} + \gamma - \log 2.
\end{align}
\end{proof}

As already noted at the end of \S \ref{signal-section}, the explicit expression
for $c_+(F(h))= G *_t h(t)$, obtained in Lemma \ref{G} by means of integral transforms,
cannot be obtained by naive expansion of $F(h)(x,t)$ in $x$ directly.

\begin{remark}\label{error}
A similar argument may be performed using Fourier transformation instead of 
the Laplace transform. The Fourier transform has been used in \cite[Lemma 4.4]{Moo:HKA},
where however the factor $(\gamma - \log 2)$ from \eqref{LG} is incorrectly missing. 
More precisely, the scaling argument outlined in \cite{Moo:HKA} determines the Fourier
and the Laplace transform of $G$ to be given by $\log \sqrt{\zeta}$ only up to an 
additive constant, which cannot be specified by her method.
\end{remark}

Returning back to \eqref{correction} and taking the Laplace transform of 
its third identity, we obtain 
\begin{equation}\label{48}
\begin{split}
&h(t) = - \mathscr{L}^{-1} \left((\mathscr{L}G) + \tan(\theta)\right)^{-1} *_t N'E_+\phi =: -\, K_{\theta} *_t N'E_+ \phi, \\
&u+F(h) = \left(E_+ + N'E_+ *_t K_{\theta} *_t N'E_+ \right) \phi \in \dom (\Delta(\theta)),
\end{split}
\end{equation}

where the inverse Laplace transform in the formula for $K_{\theta}$ is given explicitly as follows.
For any $\delta>0$ and $\kappa_{\theta}:= \gamma -\log 2 + \tan \theta$ we have a priori
\begin{align*}
K_{\theta}(t) &= \frac{1}{2\pi i} \int_{\delta + i\R} \frac{e^{t\zeta}}{\left(\log \sqrt{\zeta} + \kappa_{\theta}\right)} \, d\zeta
= \frac{1}{2\pi i} \int_{i\R} \frac{e^{t\zeta}}{\left(\log \sqrt{\zeta} + \kappa_{\theta}\right)} \, d\zeta\\
&- \lim\limits_{R\to\infty} \frac{1}{2\pi i} \int_0^\delta e^{itR}e^{xt}  \left(\log \sqrt{R} + i \arctan(x/\delta)/2+ \kappa_{\theta}\right)^{-1} \, dx\\
&+ \lim\limits_{R\to\infty} \frac{1}{2\pi i} \int_0^\delta e^{-itR}e^{xt}  \left(\log \sqrt{R} - i \arctan(x/\delta)/2+ \kappa_{\theta}\right)^{-1} \, dx.
\end{align*}
The latter two integrals behave as $O((\log R)^{-1})$, as $R\to \infty$, and hence vanish in the limit. 
Consequently
\begin{equation}
\label{K}
K_{\theta}(t) = \frac{1}{2\pi i} \int_{i\R} e^{t\zeta}  \left(\log \sqrt{\zeta} + 
\gamma - \log 2 + \tan(\theta)\right)^{-1} \, d\zeta.
\end{equation}

Well-definement of $K_{\theta} *_t N'E_+ \phi$, and regularity of $h,F(h)$ is not obvious and requires a detailed 
analysis of $K_{\theta}$, which is the content of \S \ref{auxiliary} below. Anticipating this, 
we have proved in view of \eqref{48} the following

\begin{thm}\label{4-3}
The heat kernel $E_{\theta}$ of $\Delta(\theta), \theta\in (0,\pi)$ is given by
$$
E_{\theta} = E_+ + N'E_+ *_t K_{\theta} *_t N'E_+. 
$$
\end{thm}

%%%%%%%%%%%%%%%%%%
\section{Analysis of the convolution kernel $K_{\theta}$}\label{auxiliary}
%%%%%%%%%%%%%%%%%%

The first key step here is a specific integral representation of $K_{\theta}(t)$.

\begin{prop}\label{K-integral}
Let $\kappa(\theta):=  \gamma - \log 2 + \tan(\theta)$. Then there exists 
a bounded $K^1_{\theta}\in C^\infty(\R^+_0)$, such that
\begin{equation*}
\begin{split}
K_{\theta}(t) &= \frac{(-1)}{\pi} \textup{Im} \left\{
\int_1^\infty e^{-ty}  \left(\log \sqrt{y} + i\pi/2 + \kappa(\theta)\right)^{-1} \, dy \right\} +  K^1_{\theta}(t) \\
&= 2 \int_1^\infty e^{-ty} \left((\log y+ 2\kappa(\theta))^2 + \pi^2\right)^{-1} \, dy +  K^1_{\theta}(t).
\end{split}
\end{equation*}
\end{prop}

\begin{proof} We compute from \eqref{K}
\begin{align*}
K_{\theta}(t) &= \frac{1}{2\pi i} \int_{i\R} e^{t\zeta}  \left(\log \sqrt{\zeta} + \kappa(\theta)\right)^{-1} \, d\zeta \\
&= \frac{1}{2\pi i} \int_{i[-1,1]} e^{t\zeta}  \left(\log \sqrt{\zeta} + \kappa(\theta)\right)^{-1} \, d\zeta \\
&+ \frac{1}{2\pi} \int_1^\infty e^{itx}  \left(\log \sqrt{x} + i\pi/4 + \kappa(\theta)\right)^{-1} \, dx \\
&+ \frac{1}{2\pi} \int_1^\infty e^{-itx}  \left(\log \sqrt{\zeta} -i\pi/4 + \kappa(\theta)\right)^{-1} \, dx\\
&=: \widetilde{K}^1_{\theta}(t) + K^2_{\theta}(t) + K^3_{\theta}(t).
\end{align*}
The first summand $\widetilde{K}^1_{\theta}\in C^\infty(\R^+_0)$ is smooth and bounded. For the summand 
$K^2_{\theta}$ we deform the integration contour to $i[1,\infty)$. Let the contour
$\gamma_R := \{R \exp(i\phi) \mid \phi \in [0,\pi/2]\}$ be oriented counterclockwise. 

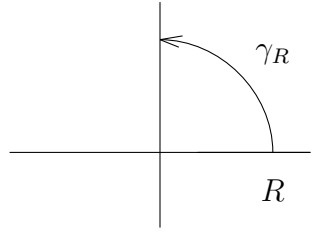
\begin{figure}[h]
\begin{center}
\begin{tikzpicture}
\draw (-2,0) -- (2,0);
\draw (0,-1) -- (0,2);

\draw (0.5,0) -- (2,0);
\draw (1.5,0) .. controls (1.5,0.8) and (0.8,1.5) .. (0,1.5);

\draw (0,1.5) -- (0.3,1.55);
\draw (0,1.5) -- (0.25,1.4);

\node at (1.5,-0.5) {$R$};
\node at (1.5,1.3) {$\gamma_R$};

\end{tikzpicture}
\end{center}
\label{contour}
\caption{The integration contour $\gamma_R$.}
\end{figure}

We use the $O$-notation for the asymptotics as $R\to \infty$.
Then, substituting $x=R \exp(i\phi)$, we find for some $C>0$
\begin{align*}
\int_{\gamma_R} \mid e^{itx}  &\left(\log \sqrt{x} + i\pi/4 + \kappa(\theta)\right)^{-1}\mid \, dx 
\\ &= R \int_0^{\pi/2} \mid\frac{\exp{(-t R\sin\phi + i(\phi + t R\cos\phi))}}
{\log \sqrt{R} + i(\phi/2 + \pi/4) + \kappa(\theta)}\mid \, d\phi  \\
&= R \int_0^{\pi/4} \mid \frac{\exp(-t R\sin\phi + i(\phi + t R\cos\phi))}
{\log \sqrt{R} + i(\phi/2 + \pi/4) + \kappa(\theta)}\mid  \, d\phi \\
& \leq \frac{CR}{\log R} \int_0^{\pi/4} \exp(-t R\sin\phi) \cos \phi \, d\phi + O(R^{-\infty}) \\
& = \frac{CR}{\log R} \int_0^{\sin \pi/4} \exp(-t R y) \, dy + O(R^{-\infty})=O\left(\frac{1}{\log R}\right),
\end{align*}
where in the fourth line we have used the fact that $\cos \phi$ is bounded 
from below for $\phi\in [0,\pi/4]$, and in the final line we substituted $y=\sin \phi$.
Consequently, we may indeed deform the integration contour of $K^2_{\theta}$ to $i[1,\infty)$ and find
\begin{align*}
K^2_{\theta}(t) &= \frac{i}{2\pi} \int_1^\infty e^{-ty}  \left(\log \sqrt{y} + i\pi/2 + \kappa(\theta)\right)^{-1} \, dy
\\ &+ \int_{\gamma_1} e^{itx}  \left(\log \sqrt{x} + i\pi/4 + \kappa(\theta)\right)^{-1} \, dx,
\end{align*}
where the second summand is clearly smooth and bounded on $\R^+_0$, since $\textup{Im}(x)\geq 0$ for $x\in \gamma_1$.
We denote the second summand in the expression for $K^2_\theta$ by $\widetilde{K}^2_\theta \in C^\infty(\R^+_0)$. 
By a similar exercise we deform the integration contour of $K^3_{\theta}$ to $(-i[1,\infty))$ and find
\begin{align*}
K^3_{\theta}(t) &= \frac{-i}{2\pi} \int_1^\infty e^{-ty}  \left(\log \sqrt{y} - i\pi/2 + \kappa(\theta)\right)^{-1} \, dy
\\ &+ \int_{\overline{\gamma_1}} e^{-itx}  \left(\log \sqrt{x} - i\pi/4 + \kappa(\theta)\right)^{-1} \, dx,
\end{align*}
where $\overline{\gamma_1}=\{\overline{z} \mid z\in \gamma_1\}$ is oriented clockwise; again the second 
summand is clearly smooth and bounded on $\R^+_0$,
since $\textup{Im}(x)\leq 0$ for $x\in \overline{\gamma_1}$.
We denote the second summand in the expression for $K^3_\theta$ by $\widetilde{K}^3_\theta \in C^\infty(\R^+_0)$. 
In total we now obtain using $i(z-\overline{z})=-2\textup{Im}(z)$ and setting 
$K^1_\theta := \widetilde{K}^1_\theta + \widetilde{K}^2_\theta + \widetilde{K}^3_\theta$
\begin{equation}
\label{5-1}
\begin{split}
K_{\theta}(t) &= \frac{(-1)}{\pi} \textup{Im} \left\{
\int_1^\infty e^{-ty}  \left(\log \sqrt{y} + i\pi/2 + \kappa(\theta)\right)^{-1} \, dy \right\} +  K^1_{\theta}(t) \\
&= 2 \int_1^\infty e^{-ty} \left((\log y+ 2\kappa(\theta))^2 + \pi^2\right)^{-1} \, dy +  K^1_{\theta}(t).
\end{split}
\end{equation}
\end{proof}

As a consequence of Proposition \ref{K-integral}, we find that $K_{\theta}\in L^1_{\textup{loc}}(\R^+_0) \cap C^\infty(\R^+)$ 
is bounded as $t\to\infty$. Integrability at $t=0$ follows using Fubini theorem. Indeed, we find for any $T>0$,
using \eqref{5-1}
\begin{align*}
& \int_1^\infty \int_0^T e^{-ty} \left((\log y+ 2\kappa(\theta))^2 + \pi^2\right)^{-1} \, dt \, dy\\
= &\int_1^\infty \left( \frac{1-e^{-Ty}}{y}\right) \left((\log y+ 2\kappa(\theta))^2 + \pi^2\right)^{-1} \, dy <\infty.
\end{align*}
On the other hand, also $N'E_+ \phi\in L^1_{\textup{loc}}(\R^+_0) \cap C^\infty(\R^+)$, bounded as $t\to\infty$.
Indeed, substituting $X=x/\sqrt{t}$ we find for $\textup{supp}(\phi)\subset [0,1)$
\begin{align*}
N'E_+ \phi(t) &=\int_0^1  \frac{\sqrt{x}}{2t} \exp\left(- \frac{x^2}{4t}\right) \phi(x)\, dx \\
&= t^{-1/4} \int_0^{1/\sqrt{t}} \frac{\sqrt{X}}{2}\exp\left(- \frac{X^2}{4}\right) \phi(X\sqrt{t})\, dX,
\end{align*} 
which shows integrability at $t=0$.
Consequently, the convolution $h=K_{\theta} *_t N'E_+ \phi$ indeed exists and $h\in C^\infty(\R^+)$ 
is continuous at $t=0$ and not growing exponentially as $t\to\infty$. 
In particular, the Laplace transform $\mathscr{L}(h)(\zeta)$ is well-defined for $\textup{Re}(\zeta)>0$. Similarly, 
for each fixed $x\in \R^+_0$, the signaling solution $F(h)(x,\cdot)\in C^\infty(\R^+)$ is continuous at $t=0$ and
its Laplace transform $\mathscr{L}F(h)(\zeta)$ is also well-defined for $\textup{Re}(\zeta)>0$.

%%%%%%%%%%%%%%%%%%
\section{Unsual heat trace expansions}\label{unusual-section}
%%%%%%%%%%%%%%%%%%

In this final section we study the asymptotic expansion of $\textup{Tr}_1 (E_{\theta}), \theta \in [0,\pi)\backslash \{\pi/2\}$ 
and establish Theorem \ref{main}. Discussion of the kernel $N'E_+ *_t K_{\theta} *_t N'E_+$
is central and consists of two steps. We begin with an explicit evaluation of asymptotics of $\textup{Tr}_1 N'E_+ *_t N'E_+$ as $t\to 0$.
\begin{prop}\label{TN}
$$\textup{Tr}_1 \, N'E_+ *_t N'E_+ \equiv \int_0^1 (N'E_+ *_t N'E_+) (x,t) \, dx
= \frac{1}{2} + O(t^{\infty}), \ t \to 0.$$ 
\end{prop} 

\begin{proof}
We first substitute $y=x^2$
\begin{align*}
\int_0^1 &(N'E_+ *_t N'E_+) (x,t) \, dx = \int_0^1 \int_0^t N'E_+(x,t-\wt) N'E_+ (x,\wt) \, d\wt\, dx\\
&= \int_0^1 \int_0^t \frac{x}{4\, \wt\, (t-\wt)} \, \exp \left(- x^2 \left( \frac{1}{4(t-\wt)} + \frac{1}{4\wt}\right)\right)\, d\wt\, dx \\
&= \int_0^1 \int_0^t \frac{1}{8\, \wt\, (t-\wt)} \, \exp \left(- y \left( \frac{1}{4(t-\wt)} + \frac{1}{4\wt}\right)\right)\, d\wt\, dy.
\end{align*}
By Fubini theorem we may interchange the integration orders and find after integration in $y$
\begin{align*}
&\int_0^1 \int_0^t \frac{1}{8\, \wt\, (t-\wt)} \, \exp \left(- y \left( \frac{1}{4(t-\wt)} - \frac{1}{4\wt}\right)\right)\, d\wt\, dy \\
&=  \frac{1}{2t} \int_0^t \left(1 -  \exp \left(- \frac{t}{4\, \wt\, (t-\wt)}\right)\right) d\wt
= \frac{1}{2} + O(t^{\infty}), \ t\to 0,
\end{align*}
where in the last step we have estimated for $\wt\in [0,t]$
$$\exp \left(- \frac{t}{4\, \wt\, (t-\wt)}\right) \leq \exp \left(-\frac{1}{4\, t}\right).$$
\end{proof}

We can now discuss the trace of full kernel $N'E_+ *_t K_{\theta} *_t N'E_+$.

\begin{thm}
\begin{align*}
&\textup{Tr}_1 \left( N'E_+ *_t K_{\theta} *_t N'E_+\right) \sim \sum_{j=0}^\infty b_j t^j \\
 &+ \frac{1}{\pi} \textup{Im} \left( \int_1^\infty \frac{e^{-ty}}{y}\left(\log \sqrt{y} + i\pi + 2\kappa(\theta)\right) \right) \, dy,
\quad  t\to 0.
\end{align*}
\end{thm}

\begin{proof}
Using Proposition \ref{K-integral} we find
\begin{align*}
&\textup{Tr}_1 \left( N'E_+ *_t K_{\theta} *_t N'E_+\right) \\
&= 2 \int_0^1 \int_0^t  \int_1^\infty \frac{e^{-(t-\wt)y}}{\left((\log y+ 2\kappa(\theta))^2 + \pi^2\right)} 
(N'E_+ *_t N'E_+)(x,\wt) \, dy \, d\wt \, dx\\
&+ \int_0^1 \int_0^t  K^1_{\theta}(t-\wt) (N'E_+ *_t N'E_+)(x,\wt)\, d\wt \, dx
=:T_1+T_2.
\end{align*}

We want to apply Fubini theorem and interchange the integration orders.
For this we note using Proposition \ref{TN}
\begin{align*}
&\int_1^\infty \int_0^t \int_0^1  \frac{e^{-(t-\wt)y}}{\left((\log y+ 2\kappa(\theta))^2 + \pi^2\right)} 
(N'E_+ *_t N'E_+)(x,\wt)  \, dx \, d\wt \, dy \\
&= \int_1^\infty \int_0^t  \frac{e^{-(t-\wt)y}}{\left((\log y+ 2\kappa(\theta))^2 + \pi^2\right)} 
\left(\frac{1}{2} + O(\wt^{\infty})\right) \, d\wt \, dy \\
&\leq \left(\frac{1}{2} + O(t^{\infty})\right)  \int_1^\infty 
\frac{1 - e^{-ty}}{y\left((\log y+ 2\kappa(\theta))^2 + \pi^2\right)} \, dy < \infty,
\end{align*}
due to $\log(y)^2$ behaviour in the denominator of the integrand.
Moreover 
\begin{align*}
 &\int_0^t \int_0^1  K^1_{\theta}(t-\wt) (N'E_+ *_t N'E_+)(x,\wt) \, dx \, d\wt  \\
&=  \int_0^t  K^1_{\theta}(t-\wt) \left(\frac{1}{2} + O(\wt^{\infty})\right) \, d\wt \\
&\leq \left(\frac{1}{2} + O(t^{\infty})\right)  \int_0^t  |K^1_{\theta}(t-\wt)| \, d\wt <\infty,
\end{align*}
due to smoothness of $K^1_{\theta}$. Consequently we may indeed interchange the integration 
orders and obtain using Proposition \ref{TN} and the integral representation of 
$K_{\theta}(t)$ in Proposition \ref{K-integral}
\begin{equation}
\label{T-1-exp}
\begin{split}
T_1 &=  \frac{(-1)}{\pi} \int_1^\infty \int_0^t \textup{Im} \left\{
e^{-(t-\wt)y}  \left(\log \sqrt{y} + i\pi/2 + \kappa(\theta)\right)^{-1} \right\} \\
&\times \int_0^1  (N'E_+ *_t N'E_+)(x,\wt) \, dx \, d\wt \, dy \\
&= \frac{1}{\pi} \int_1^\infty \textup{Im} \left\{
\frac{(e^{-ty}-1)}{y\left(\log \sqrt{y} + i\pi + 2\kappa(\theta)\right)} \right\} \, dy
 + O(t^\infty),
\end{split}
\end{equation}

as $t \to 0$. For $T_2$ we obtain
\begin{equation}
\label{T-2-exp}
\begin{split}
T_2 &= \int_0^t  K^1_{\theta}(t-\wt) \int_0^1 (N'E_+ *_t N'E_+)(x,\wt)\, dx \, d\wt \\
&= \int_0^t  K^1_{\theta}(t-\wt)  \left(\frac{1}{2} + O(\wt^{\infty})\right)  \, d\wt
\sim \sum_{j=0}^\infty c_j t^j,
\end{split}
\end{equation}
since $K^1_{\theta} \in C^\infty(\R^+_0)$.
The expansions \eqref{T-1-exp} and \eqref{T-2-exp} together yield a full
asymptotic expansion of $\textup{Tr}_1 \left( N'E_+ *_t K_{\theta} *_t N'E_+\right)$
\begin{align*}
&\textup{Tr}_1 \left( N'E_+ *_t K_{\theta} *_t N'E_+\right)  \\ &\sim
 \frac{1}{\pi} \textup{Im} \left( \int_1^\infty 
\frac{e^{-ty}}{y}\left(\log \sqrt{y} + i\pi + 2\kappa(\theta)\right) \right) \, dy
\\ &+ \sum_{j=0}^\infty b_j t^j, \quad  t\to 0.
\end{align*}
\end{proof}

Together with Theorem \ref{4-3} we obtain the statement of
Theorem \ref{main}.

\providecommand{\bysame}{\leavevmode\hbox to3em{\hrulefill}\thinspace}
\providecommand{\MR}{\relax\ifhmode\unskip\space\fi MR }
\providecommand{\MRhref}[2]{%
  \href{http://www.ams.org/mathscinet-getitem?mr=#1}{#2}
}
\providecommand{\href}[2]{#2}


\begin{thebibliography}{\textsc{FMPS03}}

\bibitem[\textsc{AbSt92}]{AbrSte:HOM}
\emph{Handbook of mathematical functions with formulas, graphs, and mathematical tables},
Edited by \textsc{M. Abramowitz} and \textsc{I. A. Stegun}, Reprint of the 1972 edition. 
Dover Publications, Inc., New York, 1992. xiv+1046 pp. ISBN: 0-486-61272-4 \MR{1225604 (94b:00012)}

\bibitem[\textsc{BrSe87}]{BruSee:RES}
\textsc{J.~Br{\"u}ning} and \textsc{R.~Seeley}, \emph{The resolvent expansion
  for second order regular singular operators}, J. Funct. Anal. \textbf{73}
  (1987), no.~2, 369--429. \MR{899656 (88g:35151)}

\bibitem[\textsc{FMPS03}]{FMPS:UPZ}
\textsc{H.~Falomir}, \textsc{M.~A. Muschietti}, \textsc{P.~A.~G. Pisani}, and
  \textsc{R.~Seeley}, \emph{Unusual poles of the {$\zeta$}-functions for some
  regular singular differential operators}, J. Phys. A \textbf{36} (2003),
  no.~39, 9991--10010. \MR{2024508 (2004k:58049)}

\bibitem[\textsc{GKM10}]{GKM:TEF}
\textsc{J. B. Gil}, \textsc{T. Krainer}, \textsc{G. A. Mendoza}, 
\emph{Trace expansions for elliptic cone operators with stationary domains},
Trans. Amer. Math. Soc. \textbf{362} (2010), no. 12, 6495--6522. 
\MR{2678984 (2011h:58040)}

\bibitem[\textsc{KLP06}]{KLP:UPR}
\textsc{K.~Kirsten}, \textsc{P.~Loya}, and \textsc{J.~Park}, \emph{The very
  unusual properties of the resolvent, heat kernel, and zeta function for the
  operator {$-d\sp 2/dr\sp 2-1/(4r\sp 2)$}}, J. Math. Phys. \textbf{47} (2006),
  no.~4, 043506, 27. \MR{2226343 (2007c:58050)}

\bibitem[\textsc{KLP08}]{KLP:EEA}
\textsc{K.~Kirsten}, \textsc{P.~Loya}, and \textsc{J.~Park}, \emph{Exotic expansions 
and pathological properties of $\zeta$-functions on conic manifolds}, J. Geom. Anal. \textbf{18} 
(2008), no. ~3, 835--888. \MR{2420767 (2009j:58051)} 

\bibitem[\textsc{Les97}]{Les:OFT}
\textsc{M.~Lesch}, \emph{Operators of {F}uchs type, conical singularities, and
  asymptotic methods}, Teubner-Texte zur Mathematik [Teubner Texts in
  Mathematics], vol. 136, B. G. Teubner Verlagsgesellschaft mbH, Stuttgart,
  1997. \MR{1449639 (98d:58174)}

\bibitem[\textsc{LeVe11}]{LesVer:RSS}
\textsc{M. Lesch}, \textsc{B. Vertman}, \emph{Regular singular Sturm-Liouville 
operators and their zeta-determinants}, J. Funct. Anal. \textbf{261} (2011), no. 2, 408--450. 
\MR{2793118 (2012g:58064)} 

\bibitem[\textsc{Moo99}]{Moo:HKA}
\textsc{E.~A. Mooers}, \emph{Heat kernel asymptotics on manifolds with conic
  singularities}, J. Anal. Math. \textbf{78} (1999), 1--36. \MR{1714065
  (2000g:58039)}

\bibitem[\textsc{Ver09}]{Ver:ZDF}
\textsc{B. Vertman}, \emph{Zeta determinants for regular-singular Laplace-type operators},
J. Math. Phys. \textbf{50} (2009), no. 8, 083515, 23 pp. \MR{2554443 (2010i:58035)} 

\end{thebibliography}
\end{document}